\documentclass{article}
\usepackage{latexsym, a4wide}
\usepackage{amsmath, rotating, color, pdfsync}
\usepackage{amsfonts,amssymb, amsthm, mathrsfs}
\usepackage{bbm}
\usepackage{wrapfig}
\usepackage{amsmath}
\usepackage{natbib}
\usepackage[intlimits]{empheq}
\usepackage{multicol} 
\usepackage{amsthm}
\newcommand\unnumberedfootnote[1]{ %
        \let\temp=\thefootnote %
        \renewcommand{\thefootnote}{}%
        \footnote{#1}%
        \let\thefootnote=\temp%
        \addtocounter{footnote}{-1}}

\newtheorem{theorem}{Theorem}
\newtheorem{proposition}{Proposition}[section]

\newtheorem{corollary}[proposition]{Corollary}

\theoremstyle{definition}
\newtheorem{remark}[proposition]{Remark}

\usepackage{yfonts, eufrak}
\usepackage[normalem]{ulem}
\numberwithin{equation}{section}

\DeclareMathAlphabet{\mathpzc}{OT1}{pzc}{m}{it}

\begin{document}
\title{\LARGE Fixation probabilities and hitting times \\for  low levels of
  frequency-dependent selection}

\thispagestyle{empty}

\author{{\sc by P. Pfaffelhuber and A. Wakolbinger} }
  \date{\today}

\maketitle
\unnumberedfootnote{\emph{AMS 2010 subject classification.} {\tt
    60J70} (Primary) {\tt 91A22, 92D15, 91A15} (Secondary).}

\unnumberedfootnote{\emph{Keywords and phrases.} Diffusion processes,
  frequency-dependent selection, Green function, evolutionary game
  theory, frequency spectrum}

\begin{abstract}
  \noindent
  In population genetics, diffusions on the unit interval are often
  used to model the frequency path of an allele. In this setting we
  derive approximations for fixation probabilities, expected hitting
  times and the expected frequency spectrum for low levels of
  frequency-dependent selection. Specifically, we rederive and extend
  the one-third rule of evolutionary game theory (Nowak et al., 2004)
  and effects of stochastic slowdown (Altrock and Traulsen,
  2009). Since similar effects are of interest in other application
  areas, we formulate our results for general one-dimensional
  diffusions.
\end{abstract}

\section{Introduction}\label{Sec1}
Our motivation for this note came from the desire to find a both
intuitive and generalizable explanation for the so-called {\em
  one-third rule} in evolutionary game theory
\citep{Nowak2004}. Phrased in the language of population genetics, the
one-third rule says: {\em Assume that~$X$ is a Wright-Fisher diffusion with
  selection coefficient $\alpha>0$ and linear frequency-dependent
  selection
\begin{equation}\label{drift}
  \psi(y) = \beta -\gamma y, \quad 0\le y\le 1,
\end{equation}
with $\beta, \gamma \in \mathbb R$}, i.e. $X$ is a $[0,1]$-valued process
$X$ satisfying the stochastic differential equation
\begin{equation}\label{WF}
  dX_t = \alpha \psi(X_t)X_t(1-X_t) + \sqrt{X_t(1-X_t)}dW_t
\end{equation} 
with $W$ being a standard Brownian motion. {\em Then the probability
  of fixation in 1 is, for small positive $\alpha$ and a small initial
  frequency $x$, larger than $x$ (which is the fixation probability in
  the neutral case $\alpha=0$; see e.g.\ (5.17) in
  \citealp{Ewens2004}) if and only if
  $\psi(\tfrac 13)>0$}.

We write $T_0$ and $T_1$ for the first times at which $X$ hits the
boundaries $0$ and $1$, respectively, and $T := T_0 \wedge T_1$ for
the first time at which $X$ hits the boundary $\{0,1\}$.
We ask the following questions: Under which conditions on the
frequency dependent selection $\psi$ -- which may then as well be more
general than given in \eqref{drift} -- is for small positive $\alpha$,
but not necessarily small $x$,
\\\\
{$\bullet$ the fixation probability $\mathbf P^\alpha_x(X_T=1)$ larger
  than $x$ (which is the fixation probability in the {\em neutral}
  case, $\alpha=0$)}
\\\\
{$\bullet$ the expected time to fixation (either unconditional or
  conditional on fixation) or extinction larger than that in the
  neutral case?}
\paragraph{Hitting probabilities.}
To illustrate our findings, let us come back to $\psi$ as given in
\eqref{drift}. We will show (see Corollary~\ref{cor1}) that
\begin{equation}
  \label{eq:change}
  \begin{aligned}
    \mathbf P^\alpha_x(X_T=1) = x + \alpha x(1-x)(\beta - \tfrac \gamma 3 (1 +
    x)) + \mathcal O(\alpha^2)
  \end{aligned}
\end{equation}
as $\alpha \to 0$.
This gives the following generalization of the one-third rule:

\smallskip \noindent
{\em For the Wright-Fisher diffusion \eqref{WF} with frequency dependent selection $\psi$ given by \eqref{drift}, the fixation probability in 1 is,
  for small positive $\alpha$ and fixed initial frequency $x$, larger
  than~$x$ (which is the fixation probability in the neutral case
  $\alpha=0$) if and only if $\psi(\tfrac {1+x}3)>0$.}
\paragraph{Fixation times.} In a series of papers
\citep{Altrock2009,Altrock2010,Altrock2012} have reported the
following -- at first sight maybe counter-intuitive -- result on what
they call a {\em stochastic slowdown effect}: Conditioned on fixation,
a selective allele can have a longer expected fixation time than a
neutral one. This effect was analyzed in the just quoted papers for a
finite population. Some structural insights, however, arise
in the appropriate diffusion limit. Thus, in Theorem~\ref{T2} we will
analyze the expected hitting time of the boundary $\{0,1\}$ as
$\alpha\to 0$. 
 In
a similar way, in Theorems~\ref{T3} and~\ref{T4} we will obtain the
approximate expected conditional hitting times of boundary points 1
and~0 as $\alpha \to 0$ when starting near $0$. For the Wright-Fisher
diffusion \eqref{WF}, Corollaries \ref{cor2}, \ref{cor3} and
\ref{cor4} specialize to
\begin{align} 
  \mathbf E^\alpha_x[T]&= 2x - 2x\log x + 2\alpha \beta x +  o(x, \alpha) \quad \mbox{ as }x,\alpha \to 0, \label{ETWF}\\
  \mathbf E_{0+}^{\alpha \ast}[T_1] &=  2 + \alpha \frac \gamma 9+  o(x, \alpha) \quad \mbox{ as }\alpha \to 0,
  \label{ET1WF}\\
  \mathbf E^\alpha_{x\ast}[T_0]&= - 2x\log x + \alpha x \frac{5\gamma}{9} +  o(x, \alpha) \quad \mbox{ as }x,\alpha \to 0, \label{ET2WF}
\end{align}
where $\mathbf E_x^{\alpha \ast}$ and $\mathbf E_{x\ast}^{\alpha}$
denote the conditional expectation $\mathbf E_x^\alpha [.|T_1<T_0]$
and $\mathbf E_x^\alpha [.|T_0<T_1]$, respectively. Notably,
\eqref{ETWF} does not depend on~$\gamma$, whereas \eqref{ET1WF} and
\eqref{ET2WF} do not depend on $\beta$, and, while $\beta$ and
$\gamma$ enter with different signs in \eqref{drift}, all signs in
\eqref{ETWF} -- \eqref{ET2WF} are positive. Intuitive
interpretations / explanations of these facts will be given in Section  \ref{Sec4}.
\paragraph{Frequency spectrum.}
Consider a population whose allele frequencies follow \eqref{WF}; see
e.g.\ \cite{Busta2001} for such a model. Assume that $\psi$ satisfies \eqref{drift}.
Let
$f^\alpha(x)dx$ be the expected number of alleles at frequency $x$ in
a model with constant immigration of new alleles. (For details, see
Theorem~\ref{T5}.). We obtain from Corollary \ref{cor5}
\begin{align}
  \label{ET3WF}
  f^\alpha(x) & = \frac 1x + \alpha(\beta + \gamma(1-2x)) +  o(\alpha) \quad \mbox{ as }\alpha \to 0,
\end{align}
i.e.\  for low levels of selection there are more alleles in low frequencies than there are in the neutral case.
\paragraph{Diploid populations.}
Classically, linear frequency dependent selection as given above also
arises in diploid populations undergoing selection, which lead to
$\psi(x) = h + x(1-2h)$, and $h$ is called the {\em dominance
  coefficient}. If $h=0$, the diffusion models the frequency path of a
selected recessive allele, whereas the allele is dominant for
$h=1$. In the case $h\in(0,1)$, we speak of incomplete
dominance. Overdominance refers to $h>1$, and means (if $\alpha>0$)
that the heterozygote is fitter than any homozygote. Finally,
underdominance refers to $h<0$, and implies that the homozygote is
less fit (if $\alpha>0$) than any homozygote.

By setting $\beta=h$ and $\gamma = 2h-1$, \eqref{ETWF} then implies
the (somewhat counter-intuitive) result that fixation or extinction of
a positively (i.e.\ $\alpha>0$) selected allele takes longer than
under neutrality for $h>1/2$. Moreover, \eqref{ET1WF} gives that --
conditional on fixation -- the positively selected allele takes longer
to fix than under neutrality if $h>1/2$. The latter result was shown
already by \cite{LachmannMafessoni2015}. In addition, they report that
mildly deleterious (i.e.\ $\alpha<0$) recessive (i.e.\ $h=0$) alleles
on average survive in a population slightly longer than neutral ones,
before getting lost. This is a direct consequence of \eqref{ET2WF}.

~~

This paper is organized as follows: In Section~\ref{Sec2}, we derive
an approximation for the fixation probability (Theorem~\ref{T1}) for
small $\alpha$. Unconditional hitting times (Theorem~\ref{T2}) as well
as conditional hitting times (Theorems~\ref{T3} and~\ref{T4}) are
treated as well, also in the situation of a more general
one-dimensional diffusion \eqref{eq:SDE}. Finally, we compute the
effect of low levels of frequency-dependent selection on the frequency
spectrum (Theorem~\ref{T5}). All our theorems come with corollaries
which treat the special case, where $X$ is a solution of \eqref{WF}
and $\psi$ is a {\em polynomial} describing the frequency
dependence. Such a polynomial frequency-dependence is the topic of
Section~3, which starts by recalling how the Wright-Fisher dynamics
\eqref{WF} arises as a scaling limit of an evolutionary game. We will
discuss evolutionary games in a haploid population, give the
generalization of the one-third rule as described above, and connect
our results to the effect of stochastic slowdown. In addition, we will
review a diploid situation considered in \cite{Hashimoto2009}, and
show how this leads to a frequency dependence given by $\psi$ being a
polynomial of degree~3. Our Theorem will then directly render (and
explain) the 2/5 and 3/10 rules discovered in
\cite{Hashimoto2009}. Section~\ref{Sec4} contains a discussion and
some more implications of our findings.  In Appendix~\ref{Sec5}, we
give the proofs to all our main results from Section~\ref{Sec2}.

\section{Main results}\label{Sec2}
In this section, we are concerned with the following situation.  Let
$\sigma: [0,1] \mapsto \mathbb R_+$ have a continuous derivative, and
let $\psi: [0,1] \mapsto \mathbb R$ be such that
$\mu:= \psi \cdot \sigma^2$ is bounded. For $\alpha \in \mathbb R$,
let $X$ under the measure $\mathbf P^\alpha_x$ be the It\^o diffusion
started in $X_0=x \in (0,1)$ and solving
\begin{align}
  \label{eq:SDE}
  dX = \alpha \mu(X) dt + \sigma(X) dW
\end{align}
up to the first hitting time $T = T_0 \wedge T_1$ of $\{0,1\}$. A most
important special case is that $\sigma^2(y) = y(1-y)$, which brings us
back to the Wright-Fisher diffusion with frequency dependent
selection, and that $\psi$ is a polynomial.  However, we note that our
results also apply to more general $\sigma^2$ and~$\psi$.  We
formulate our results for the state space $[0,1]$ (also because of
notational convenience); versions for more general state spaces
$[u,v]$ are easily obtained by scaling.

Our first result is on fixation probabilities for low levels of selection, and can be seen as a generalization of the $1/3$-rule described after
\eqref{eq:change}.
\begin{theorem}[Hitting probabilities\label{T1}]
  We have
  \begin{align}\label{mainres}
    \frac{1}{2\alpha} 
    \Big( \mathbf P^\alpha_x(T_1 < T_0) - x\Big) 
    & 
      \xrightarrow{\alpha\to 0} x\int_0^1(1-y)\psi(y) dy - \int_0^x 
      (x-y)\psi(y) dy.
  \end{align}
\end{theorem}
Next we specialize this to monomial $\psi$; by linearity of the
r.h.s. of \eqref{mainres} in $\psi$, this then immediately extends to
polynomial $\psi$.
\begin{corollary}[Hitting probabilities for polynomial
  $\psi$\label{cor1}]
  Let $k=0,1,2,...$ 
  \begin{enumerate}
  \item If $\psi(x) = x^k$, then
    \begin{align*}
      \frac{\partial}{\partial \alpha}\mathbf P^\alpha_x(T_1 < T_0) = \frac{1}{\binom{k+2}{2}}x(1-x^{k+1}) .
    \end{align*}
    In particular, for $k=0$,
    \begin{align*}
     \frac{\partial}{\partial \alpha} \mathbf P^\alpha_x(T_1 < T_0) = x(1-x),
    \end{align*}
    while for $k=1$, 
    \begin{align*}
     \frac{\partial}{\partial \alpha} \mathbf P^\alpha_x(T_1 < T_0) = \tfrac 13 x(1-x^2).
    \end{align*}
  \item If $\psi(x) = (1-x)^k$, then
    \begin{align*}
      \frac{\partial}{\partial \alpha}\mathbf P^\alpha_x(T_1 < T_0) = \frac{1}{\binom{k+2}{2}}(1-x)(1-(1-x)^{k+1}).
    \end{align*}
  \end{enumerate}
\end{corollary}

\noindent
We now turn to the analysis of fixation times. First, we are dealing
with the unconditional case, i.e.\ the the expectation of the hitting
time $T$. Note that, for $\psi$ as in \eqref{drift},
Corollary~\ref{cor2} specializes to \eqref{ETWF} in the introduction.

\begin{theorem}[Expected hitting time --
  unconditional case\label{T2}]
  Consider the same situation as in Theorem~\ref{T1}. Then, if all
  integrals exist,
  \begin{equation}
    \label{eq:T2}
    \begin{aligned}
      \lim_{\alpha\to 0} \frac{1}{4\alpha} \Big(\frac 1x\Big( \mathbf
      E_x^\alpha[T] & - 2\int_0^x \frac{y}{\sigma^2(y)} dy -
      2x\int_x^1 \frac{(1-y)}{\sigma^2(y)} dy\Big) \Big) \\ &
      \xrightarrow{x\to 0} \int_0^1\frac 1{\sigma^2(y)}
      \Big((1-y)\int_0^y (1-z)\psi(z)dz - y \int_y^1
      (1-z)\psi(z)dz\Big) dy.
    \end{aligned}
  \end{equation}
\end{theorem}

\begin{corollary}[Expected hitting time -- unconditional case -- polynomial $\psi$\label{cor2}]
  Let $k=0,1,2,...$
  \begin{enumerate}
  \item If $\psi(x) = x^k$, 
    \begin{align}\label{uTb}
      \frac{\partial}{\partial \alpha}\mathbf E^\alpha_x[T]\big |_{\alpha=0}= 4x\int_0^1 \frac{1}{\sigma^2(y)} \Big(\frac{(1-y)y^{k+1}}{k+2} - \frac{y(1-y^k)}{(k+1)(k+2)}\Big) dy + o(x) \quad \mbox{ as } x\to 0\, .    
      \end{align}
    In particular, for $k=0$,
    \begin{align}\label{uTc}
      \frac{\partial}{\partial \alpha}\mathbf E^\alpha_x[T]\big |_{\alpha=0}= 2x \int_0^1 \frac{y(1-y)}{\sigma^2(y)}dy + o(x)\quad\mbox{ as } x\to 0\, 
    \end{align}
    and for $k=1$
    \begin{align}\label{uTl}
      \frac{\partial}{\partial \alpha}\mathbf E^\alpha_x[T]\big |_{\alpha=0}= \frac {2x}3 \int_0^1 \frac{y(1-y)(2y-1)}{\sigma^2(y)}dy + o(x)\quad\mbox{ as } x\to 0\, .
          \end{align}
  \item If $\psi(x) = (1-x)^k$, 
    \begin{align*}
      \frac{\partial}{\partial \alpha}\mathbf E^\alpha_x[T]\big |_{\alpha=0}= \frac{4x}{k+2} \int_0^1 \frac{1}{\sigma^2(y)}(1-y)(1-(1-y)^{k+1})dy + o(x) \quad\mbox{ as } x\to 0\, .
    \end{align*}
  \end{enumerate}
\end{corollary}


\bigskip
\noindent
We now study the conditional hitting times in two versions, which
directly lead to \eqref{ET1WF} and \eqref{ET2WF} in the
introduction. First, the process exits at the boundary point opposite
to where it entered, and second the process exits at the same boundary
point. Here
$\mathbf E_{0+}^{\alpha \ast}[.] := \lim_{x\to 0} \mathbf
E_x^\alpha[.|T_1<T_0]$ denotes the expectation under the measure of
the diffusion started in $0$ and conditioned to reach $1$, and
$\mathbf E_{x\ast}^{\alpha}[.] := \mathbf E_x^\alpha[.|T_0<T_1]$
denotes the expectation under the measure of the diffusion started in
$x$ and conditioned to reach $0$ before $1$.

\begin{theorem}[Expected hitting time - conditional case I\label{T3}]
  In the situation of Theorem~\ref{T1}, if all integrals exist, 
  \begin{equation}
    \label{eq:T3}
    \begin{aligned}
      \lim_{\alpha\to 0} \frac{1}{4\alpha}\Big(\mathbf E_{0+}^{\alpha
        \ast} [T_1] &- 2\int_0^1 \frac{(1-y)y}{\sigma^2(y)} dy\Big)
      \\&= \int_0^1 \frac 1{\sigma^2(y)} \Big( (1-y)^2 \int_0^y
      z\psi(z) dz -y^2 \int_y^1 (1-z)\psi(z) dz\Big) dy.
    \end{aligned}
  \end{equation}
 \end{theorem}
\bigskip
\begin{corollary}[Expected hitting time -- conditional case I -- polynomial
  $\psi$\label{cor3}]
  Let $k=0,1,2,...$
  \begin{enumerate}
  \item If $\psi(x) = x^k$,
    \begin{align}\label{eq0:cor3}
      \frac{\partial}{\partial \alpha} \mathbf E_{0+}^{\alpha \ast} 
        [T_1] \big |_{\alpha=0}
      = 4\int_0^1 \frac{1}{\sigma^2(y)} \Big(\frac{(1-y)y^{k+2}}{k+2} - \frac{y^2(1-y^{k+1})}{(k+1)(k+2)}\Big)dy.
    \end{align}
    In particular, for $k=0$,
    \begin{align}\label{eq:cor3}
     \frac{\partial}{\partial \alpha} \mathbf E_{0+}^{\alpha \ast} 
        [T_1] \big |_{\alpha=0} = 0.
    \end{align}
    and for $k=1$
    \begin{align}\label{eq2:cor3}
     \frac{\partial}{\partial \alpha} \mathbf E_{0+}^{\alpha \ast} 
        [T_1] \big |_{\alpha=0} = -\frac 23 \int_0^1 \frac{y^2(1-y)^2}{\sigma^2(y)}dy.
    \end{align}
  \item If $\psi(x) = (1-x)^k$,
    \begin{align*}
      \frac{\partial}{\partial \alpha} \mathbf E_{0+}^{\alpha \ast} 
        [T_1] \big |_{\alpha=0} =
       -4 \int_0^1 \frac{1}{\sigma^2(y)} \Big(\frac{y(1-y)^{k+2}}{k+2} - 
           \frac{(1-y)^2(1-(1-y)^{k+1})}{(k+1)(k+2)}\Big)dy.
    \end{align*}
  \end{enumerate}
\end{corollary}



\noindent
We note that -- according to (11) in \cite{Griffiths2003} -- the
expected age of an allele which is at frequency $x$ is given by
$a(x) = \mathbf E^\alpha_{x\ast}[T_0]$.
Therefore, our next result is also a statement on average ages of
alleles for low levels of selection.

\begin{theorem}[Expected hitting time --
  conditional case II\label{T4}]
  In the situation of Theorem~\ref{T1}, and if all integrals exist
  \begin{equation}
    \label{eq:T4}
    \begin{aligned}
      \lim_{\alpha\to 0} & \frac{1}{4\alpha} \Big(\frac{1}{x} \mathbf
      E^\alpha_{x\ast}[T_0] - \frac 2x \int_0^x \frac{y}{\sigma^2(y)}
      dy - 2\int_x^1 \frac{(1-y)^2}{\sigma^2(y)}dy\Big) \\ &
      \xrightarrow{x\to 0}\int_0^1 \frac{1}{\sigma^2(y)} \Big(
      (1-y)^2 \int_0^y (1-2z)\psi(z)dz - 2 y(1-y)\int_y^1(1-z)
      \psi(z)dz\Big) dy.
    \end{aligned}
  \end{equation}
\end{theorem}
\bigskip
\begin{corollary}[Expected hitting time -- conditional case II -- polynomial
  $\psi$\label{cor4}]
  Let $k=0,1,2,...$
  \begin{enumerate}
  \item If $\psi(x) = x^k$,
    \begin{align}\label{eq0:cor4}
      \frac{\partial}{\partial \alpha} \mathbf E_{x\ast}^{\alpha} 
      [T_0] \big |_{\alpha=0}
      = 
      4x\int_0^1 \frac{y(1-y)}{\sigma^2(y)} \Big(\frac{k(1-y^{k+1})}{(k+1)(k+2)} 
      - \frac{1 - y^k}{k+1}\Big)dy + o(x) \qquad \text{ as $x\to 0$}.
    \end{align}
    In particular, for $k=0$,
    \begin{align}\label{eq:cor4}
      \frac{\partial}{\partial \alpha} \mathbf E_{x\ast}^{\alpha} 
      [T_0] \big |_{\alpha=0} = o(x) \qquad \text{ as $x\to 0$}
    \end{align}
    and for $k=1$
    \begin{align}\label{eq2:cor4}
      \frac{\partial}{\partial \alpha} \mathbf E_{x\ast}^{\alpha} 
      [T_0] \big |_{\alpha=0} = -\frac 23 x\int_0^1 
      \frac{(1-y)^2(1-(1-y)^2)}{\sigma^2(y)}dy  + o(x) \qquad \text{ as $x\to 0$}.
    \end{align}
  \item If $\psi(x) = (1-x)^k$,
    \begin{align*}
      \frac{\partial}{\partial \alpha} \mathbf E_{x\ast}^{\alpha} 
      [T_0] \big |_{\alpha=0} =
      4 x\int_0^1 \frac{(1-y)^2(1-(1-y)^{k+1})}{\sigma^2(y)}\frac{k}{(k+1)(k+2)}dy + o(x) \qquad \text{ as $x\to 0$}.
    \end{align*}
  \end{enumerate}
\end{corollary}

Another classical application of diffusion theory in population
genetics is the frequency spectrum. The underlying idea is that new
mutations arise with an intensity $\theta dt$ (more precisely, in the
limit $\varepsilon\downarrow 0$, a Poisson stream of immigrants of
frequency $\varepsilon$ comes in with immigration intensity
$\tfrac\theta\varepsilon dt$). The frequency path of each mutation
follows the diffusion $X$. Then, let $\theta f^\alpha(x)dx$ be the
expected number of diffusion paths in frequency $x$. According to
Theorem 7.20 of \cite{Durrett2008}, in our situation this is given by
$$ f^\alpha(x) := \frac{e^{2 \alpha \int_0^x \psi(y) dy}}{\sigma^2(x)} \mathbf P^\alpha_x(T_0<T_1).$$
This function we study next.

\begin{theorem}[Frequency spectrum\label{T5}]
  The function $f^\alpha$ defined above satisfies
  \begin{align*}
    \lim_{\alpha\to 0} \frac{1}{2\alpha}\Big(f^\alpha(x) - \frac{1-x}{\sigma^2(x)} \Big) = \frac{x}{\sigma^2(x)}\Big(\frac 1x \int_0^x (1-y)\psi(y)dy - \int_0^1(1-y)\psi(y) dy\Big).
  \end{align*}
\end{theorem}

\begin{corollary}[Frequency spectrum -- polynomial $\psi$\label{cor5}]
  Let $k=0,1,2,...$
  \begin{enumerate}
  \item If $\psi(x) = x^k$,
    \begin{align}\label{eq0:cor5}
      \frac{\partial}{\partial \alpha} f^\alpha(x)\big |_{\alpha=0}
      = - \frac{2x}{\sigma^2(x)} \frac{1 - x^k(1 + (1-x)(k+1))}{(k+1)(k+2)}.
    \end{align}
    In particular, for $k=0$,
    \begin{align}\label{eq:cor5}
      \frac{\partial}{\partial \alpha} f^\alpha(x) \big |_{\alpha=0} = \frac{x(1-x)}{\sigma^2(x)}
    \end{align}
    and for $k=1$
    \begin{align}\label{eq2:cor5}
      \frac{\partial}{\partial \alpha} f^\alpha(x) \big |_{\alpha=0} = -\frac{x(1-x)(1-2x)}{3\sigma^2(x)}
    \end{align}
  \item If $\psi(x) = (1-x)^k$,
    \begin{align*}
      \frac{\partial}{\partial \alpha} f^\alpha(x)\big |_{\alpha=0}
      = - \frac{2x}{\sigma^2(x)} \frac{1}{k+2}\Big(1-\frac 1x (1-(1-x)^{k+2}) \Big)
    \end{align*}
  \end{enumerate}
\end{corollary}

\section{Applications in Evolutionary Game Theory}
\label{sec3}
Limits of weak selection and large population size have been studied also in evolutionary game theory; see e.g.\ \cite{SampleAllen2017} for a recent discussion on
this topic. These limit arises as follows: Each individual in the
population has a certain genotype. At some high rate, pairs of
individuals are chosen at random, and the first individual imposes its
genotype upon the second. In addition to these {\em neutral} events,
which in the limit of large populations would give rise to a
Wright-Fisher diffusion, also {\em selective} events happen at a lower
rate. For this, consider two different strategies $S_1$ and $S_2$ such
that genotype $g$ has a probability $p_g$ to play strategy $S_1$ and
$1-p_g$ to play strategy $S_2$.

In order to determine the {\em fitness} of a genotype, consider the
payoff matrix
\begin{center}
  \begin{tabular}{c|cc}
    & $S_1$ & $S_2$ \\\hline
    $S_1$ & $a$ & $b$ \\
    $S_2$ & $c$ & $d$
  \end{tabular}
\end{center}
The absolute fitness of genotype $g$ is then proportional to the
average payoff it receives upon playing against a random individual
from the population.

\subsection{Evolutionary games in haploid populations}
In haploid populations, assume that there are two genotypes, $A$ and
$B$, and $A$ always plays strategy $S_1$ whereas $B$-individuals play
$S_2$. Then, if $x$ is the frequency of $A$-alleles, the fitness of
any $A$-individual is $1 + \alpha(xa + (1-x)b)$, since this individual
receives payoff $a$ if it plays against $S_1$ and $b$ if it plays
against $S_2$. By the same argument, the fitness of a $B$-individual
is $1 + \alpha(xc + (1-x)d)$. In the diffusion limit that is familiar
from population genetics (see e.g. \citealp{Ewens2004}), the relative
frequency $X$ of type $A$-individuals then follows the dynamics
\begin{equation}
  \label{eq:haploid}
  \begin{aligned}
    dX & = \alpha X(1-X)(Xa + (1-X)b - Xc - (1-X)d)dt + \sqrt{X(1-X)}dW
    \\ & = \alpha X(1-X)(\beta - \gamma X)dt + \sqrt{X(1-X)}dW
  \end{aligned}
\end{equation}
for $\beta = b-d, \gamma = b-d+c-a$. 

Alternatively, argue as follows: Each (ordered) pair of $A$
individuals plays ``selective encounters'' against each other at rate
$\alpha a$, and the first individual has an offspring which replaces a
randomly chosen individual from the population. At rate $\alpha b$, a
pair $(A,B)$ does the same, as well as at rate $\alpha c$ for $(B,A)$,
and a pair $(B,B)$ does the same at rate $\alpha d$. Using this model
we see that the frequency of $A$ increases if the first individual in
the playing pair is $A$ and the replaced individual is
$B$. Conversely, the frequency of $A$ decreases if the first
individual in the pair is $B$ and the replaced is $A$. Again, in the
appropriate scaling limit, this gives rise to \eqref{eq:haploid}.

In this situation, we can apply Theorems~\ref{T1}, \ref{T2}, \ref{T3},
\ref{T4}, \ref{T5} and their corollaries to obtain \eqref{eq:change},
\eqref{ETWF}, \eqref{ET1WF}, \eqref{ET2WF} and \eqref{ET3WF}.  Note
that we have $\sigma^2(y) = y(1-y)$ and $\psi(z) = \beta - \gamma
z$. Since all right hand sides of Theorems~\ref{T1}--\ref{T5} are
linear in $\psi$, we can directly use Corollaries \ref{cor1},
\ref{cor2}, \ref{cor3}, \ref{cor4} and \ref{cor5} and sum $\beta$
times the term for $\mu(x) = \sigma^2(x)$ and $-\gamma$ times the term
for $\mu(x) = x\sigma^2(x)$. This directly shows our claims.

In the limit $x\to0$, \eqref{eq:change} is the classical one-third
rule by \cite{Nowak2004}.  Moreover, note that the right hand side of
the unconditioned expectation in \eqref{ETWF} does not depend on
$\gamma$ (compare with (22) in \citealp{Altrock2009}), while the right
hand side of \eqref{ET1WF} does not depend on $\beta$ (compare with
(24) in \citealp{Altrock2009}). In particular, for small $\alpha$, we
see that in the unconditioned case the selected allele fixes slower
than a neutral allele iff $\beta>0$, while conditional on fixation,
fixation is slower iff $\gamma>0$.

\subsection{Evolutionary games in diploid populations}
In \cite{Hashimoto2009}, evolutionary games in a diploid population
were studied. Here, genotypes $AA, AB$ and $BB$ are formed from the
haploids using the Hardy-Weinberg equilibrium (i.e.\ genotype $AA$ has
a frequency of $x^2$, if $A$ has frequency $x$, etc.) For computing the
fitness, we consider two different cases. In the case of a dominant
$A$ allele, we have that  $AA$ as well as $AB$ play strategy $S_1$ and
$BB$ plays strategy $S_2$. 

For a dominant $A$-allele, the fitness advantage of $A$ is then
computed by assuming that it forms a genotype by randomly choosing a
mate and then having an average payoff
$\alpha((1-x)^2 b + (1-(1-x)^2)a)$.  For $B$, the same argument leads
to
$\alpha(x((1-x)^2 b + (1-(1-x)^2)a) + (1-x)((1-x)^2 d +
(1-(1-x)^2c))$.  (The $B$ allele can form $AB$ with probability $x$
and play strategy $S_1$ or form $BB$ with probability $1-x$ and play
strategy $S_2$.) In total, we find that in the diffusion limit the
frequency of $A$ follows
\begin{align*}
  dX & = \alpha X(1-X)^2 ((1-X)^2 b + (1-(1-X)^2)a - (1-X)^2 d - (1-(1-X)^2)c)dt \\ &\phantom{AAAAAAAAAAAAAAAAAAAAAAAAAAAAAAAAAAAAAAAAA} + \sqrt{X(1-X)}dW
  \\ & = \alpha X(1-X)^2 (\beta - \gamma (1-X)^2) dt + \sqrt{X(1-X)}dW
\end{align*}
for $\beta = a-c, \gamma = a-c+d-b$. Plugging
$\psi(x) = (1-x)(\beta - \gamma(1-x)^2)$ into Corollaries \ref{cor1},
\ref{cor2}, \ref{cor3}, \ref{cor4} and \ref{cor5}, straightforward
calculations give
\begin{align}\label{eq:dipDom1}
  \frac{\partial}{\partial \alpha}\mathbf P^\alpha_x(T_1 \leq T_0)\big |_{\alpha=0}  &= \frac 23x \Big(\beta - \frac 35 \gamma\Big)+o(x)\quad \mbox{ as }x \to 0,
  \\ \label{eq:dipDom2}  \frac{\partial}{\partial \alpha}\mathbf E^\alpha_x[T] \big |_{\alpha=0}
                                                                                     &= \frac{1}{3}x (6 \beta - 5 \gamma)+o(x)\quad \mbox{ as }x \to 0,
  \\\label{eq:dipDom3}
  \frac{d}{d \alpha}\mathbf E_{0+}^{\alpha \ast}[T]\big |_{\alpha=0}&= \frac{1}{3}\Big(\frac 13 \beta - \frac{29}{100}\gamma\Big),
  \\\label{eq:dipDom4}
  \frac{d}{d \alpha} \frac 1x \mathbf E_{x\ast}^{\alpha}[T_0]\big |_{\alpha=0}& = \frac 59 \beta - \frac{77}{100} \gamma + o(x)\quad \mbox{ as }x \to 0,
  \\\label{eq:dipDom5}
  \frac{d}{d \alpha} f^\alpha(x)\big |_{\alpha=0}& = \beta\Big(\frac 43 - \frac 23 x\Big) - \gamma\Big(\frac 85 - \frac{12}{5}x + \frac 85 x^2 - \frac 25x^3\Big).
\end{align}
In the case of a recessive $A$-allele, the heterozygote $AB$ plays
strategy $S_2$. In this case, the fitness advantage of $A$ is then
$\alpha(x(x^2a + (1-x^2)b) + (1-x)(x^2 c + (1-x^2)d))$, whereas the fitness advantage of
$B$ is $\alpha(x^2 c + (1-x^2)d)$. In total, $X$ follows
\begin{align*}
  dX & = \alpha X^2(1-X)( X^2 a + (1-X^2)b - X^2 c - (1-X^2)d)dt + \sqrt{X(1-X)}dW
  \\ & = \alpha X^2(1-X)(\beta - \gamma X^2)dt + \sqrt{X(1-X)}dW
\end{align*}
for $\beta = b-d, \gamma = b-d+c-a$.  From Corollaries \ref{cor1},
\ref{cor2}, \ref{cor3}, \ref{cor4} and \ref{cor5}, used with
$\psi(x) = x(\beta - \gamma x^2)$, we obtain,
\begin{align}\label{eq:dipRec1}
   \frac{\partial}{\partial \alpha}\mathbf P^\alpha_x(T_1 \leq T_0)\big |_{\alpha=0}
   &= \frac 13x \Big(\beta - \frac 3{10} \gamma\Big)+o(x)\quad \mbox{ as }x \to 0,
  \\ \label{eq:dipRec2}  \frac{\partial}{\partial \alpha} \mathbf E^\alpha_x[T]\big |_{\alpha=0}
 &=\frac{1}{6}x \gamma +o(x)\quad \mbox{ as }x \to 0,
  \\\label{eq:dipRec3}
  \frac{d}{d \alpha}\mathbf E_{0+}^{\alpha \ast}[T]\big |_{\alpha=0}
   &= -\frac{1}{3}\Big(\frac 13 \beta - \frac{29}{100}\gamma\Big),
  \\\label{eq:dipRec4}
  \frac{d}{d \alpha} \frac 1x \mathbf E_{x\ast}^{\alpha}[T_0]\big |_{\alpha=0} 
   &= -\frac 59 \beta + \frac{27}{100} \gamma + o(x)\quad \mbox{ as }x \to 0,
  \\\label{eq:dipRec5}
  \frac{d}{d \alpha} f^\alpha(x)\big |_{\alpha=0}& = -\frac 13\beta\Big(1 - \frac{1}{200}x\Big) + \frac{1}{10}\gamma\Big(1 + \frac{1}{9}x + x^2 - 4x^3\Big).
\end{align}
Equations \eqref{eq:dipDom1} and \eqref{eq:dipRec1} correspond to (and explain) the so-called  2/5 and 3/10 rules
in \cite{Hashimoto2009}.
\section{Discussion}\label{Sec4}
As described in Sections~\ref{Sec1} and \ref{sec3}, the results from
Section~\ref{Sec2} unify from the perspective of a diffusion
approximation a number of recent findings, by \cite{Nowak2004} on the
1/3 rule, by \cite{Altrock2009}, \cite{Altrock2010, Altrock2012} on
stochastic slowdown, by \cite{Hashimoto2009} on the 2/5 and 3/10 rules
in evolutionary games, and by \cite{LachmannMafessoni2015} on
``selective strolls''. In addition, they lead to the approximate
expected frequency spectrum \eqref{ET3WF} for low levels of selection.
  
A central tool for proving the results stated in Section~\ref{Sec2} is
the Green function $G^\alpha(x,y)$ of the process~$X$ given by
\eqref{eq:SDE}; see also \eqref{eq:8313}--\eqref{eq:8315} below.
Recall that the occupation measure $G^\alpha(x,y)dy$ is the expected
amount of time which $X$ (started in $x$) spends in $dy$ before time
$T$.

\paragraph{Hitting probabilities.}  Here we describe how to quickly
arrive at \eqref{eq:change} and, more generally, at the assertion of
Theorem \ref{T1}, by using the Green function. This is based on ideas
of \citep{Rousset2003} and \citep{Ladret2007} for a time-discrete
situation, and becomes even more elegant in the diffusion setting. We
first note that
\begin{equation}
  \label{eq:change1}
  \begin{aligned}
    \mathbf P^\alpha_x(X_T=1)  &
     = \mathbf E^\alpha_x[X_T] \\
   & = x +
    \int_0^\infty \frac{d}{dt} \mathbf E^\alpha_x[X_t]dt \\  &= x + 
    \int_0^\infty \mathbf E^\alpha_x[\alpha\mu(X_t)]dt \\ & = x + \alpha
    \int_0^1 G^\alpha(x,y)\mu(y) dy 
    \\Ê&= x + \alpha \int_0^1
    G^0(x,y)\mu(y) dy + \mathcal O(\alpha^2).
   \end{aligned}
\end{equation}   
For the last equality, we note (see the beginning of
Appendix~\ref{Sec5}) that $G^\alpha = G^0 +\mathcal O(\alpha)$ and
that
\begin{align*}
  G^0(x,y) & = \begin{cases} \displaystyle 2x(1-y)\frac 1{\sigma^2(y)}, & 0\leq x\leq y\leq 1,
    \\[2ex]  \displaystyle 2(1-x)y\frac 1{\sigma^2(y)}, & 0\leq y\leq x\le 1.\end{cases}
\end{align*}
Hence we may  continue \eqref{eq:change1} as
\begin{equation*}
  \begin{aligned}  
    = x + 2\alpha \Big( (1-x)\int_0^x y\psi(y) dy + x\int_x^1
    (1-y)\psi(y) dy\Big)+ \mathcal O(\alpha^2),
  \end{aligned}
\end{equation*}
which proves Theorem \ref{T1}. Specializing to \eqref{drift} (see also
Corollary \ref{cor1}) gives \eqref{eq:change} and thus the generalized
1/3 rule.

\paragraph{Unconditional hitting times.} 
Since $\mathbf E_x^\alpha[T] = \int_0^1 G^\alpha(x,y)dy$, the Green
function plays a central role for Theorem~\ref{T2}; indeed, the
assertion there can be understood as a statement on the
asymptotics of
$\frac{\partial}{\partial \alpha}\mathbf E^\alpha_x[T]\big
|_{\alpha=0} = \int_0^1\frac{\partial}{\partial \alpha} G^\alpha(x,y)
\Big |_{\alpha=0} dy $ as $x \to 0$. The proof of Theorem~\ref{T2}
(see the right hand side of \eqref{eq:green}) tells the following
refinement of \eqref{eq:T2}:
\begin{align}
  \label{eq:ref1}
  \frac 1x \frac{\partial}{\partial\alpha} G^\alpha(x,y)\Big|_{\alpha = 0} \xrightarrow{x\to 0} \frac 4{\sigma^2(y)}
  \Big((1-y)\int_0^y (1-z)\psi(z)dz - y \int_y^1 (1-z)\psi(z)dz\Big).
\end{align}
Setting $\psi=1$, this gives (compare with \eqref{uTc})
$$\frac{\partial}{\partial \alpha} G^\alpha(x,y) \Big |_{\alpha=0} = 2x  \frac{y(1-y)}{\sigma^2(y)} +o(x)\quad\mbox{ as } x\to 0\, .$$
 For an intuitive
interpretation, note that as long as $\alpha$ is positive there is a
probability of escaping a quick hitting of $0$, which for small
$\alpha$ contributes to the occupation measure $G^\alpha(x,y) \, dy$
in a way that is asymptotically proportional to the neutral case. In particular, hitting the boundary on average takes longer for small
positive $\alpha$ than under neutrality, i.e.~for $\alpha=0$. 

For the case $\psi(x) = x^k$, the integrand in \eqref{uTb} can be
written as
$$ \frac 1x \frac{\partial}{\partial\alpha} G^\alpha(x,y)\Big|_{\alpha = 0} \xrightarrow{x\to 0} 
\frac{4y(1-y)}{\sigma^2(y)}\frac{1}{(k+1)(k+2)}\Big(k+1)y^k -
\sum_{i=0}^{k-1} y^i\Big).$$ (with the convention that the sum over
the empty set is 0).

Considering the case $\sigma^2(x) = x(1-x)$ as in \eqref{WF} and
$k=1$, we note that the right hand side is {\em antisymmetric} around
$1/2$. The leading term in this difference comes from the paths which,
after starting near $0$, escape a quick hitting of $0$; the
antisymmetry around $1/2$ has to be attributed to the linearity of
$\psi$ and leads to a vanishing right hand side. For $k>1$, we find by
integrating the right hand side that
\begin{align*}
  \frac 1x \frac{\partial}{\partial \alpha}\mathbf E^\alpha_x[T] \xrightarrow{x\to 0} - \frac{4}{(k+1)(k+2)} \sum_{i=2}^{k}\frac 1i.
\end{align*}
In particular, the expected hitting time for small positive $\alpha$
is smaller than under neutrality.

\paragraph{Conditional hitting times.}
As for unconditional hitting times, the assertions of
Theorems~\ref{T3} and \ref{T4} are true even on the level of Green
functions. Conditional under $T_1 < T_0$, and writing $G^\ast$ for the
Green function in this case, we have (see formula \eqref{eq:dsj} in
the Appendix)
\begin{align*}
  \frac{\partial}{\partial\alpha} G^\ast(0+,y)\Big|_{\alpha = 0} = \frac{4}{\sigma^2(y)}
  \Big( (1-y)^2 \int_0^y
  z\psi(z) dz -y^2 \int_y^1 (1-z)\psi(z) dz\Big).  
\end{align*}
Specializing to $\psi(x) = x^k$, this gives (as a reformulation of the
integrand in \eqref{eq0:cor3})
\begin{align*}
  \frac{\partial}{\partial\alpha} G^\ast(0+,y)\Big|_{\alpha = 0} = \frac{4y^2(1-y)}{\sigma^2(y)} \frac{1}{(k+1)(k+2)}
  \Big((k+1)y^k - \sum_{i=0}^k y^i\Big)\leq 0
\end{align*}
with equality if and only if $k=0$. In particular, the expected
hitting time of~1 decreases with $\alpha$ for small $\alpha$ and small
initial value as long as $k\geq 1$. In Remark \ref{mustar} we will see
a still finer result: again for small $\alpha$ and small $x$,
conditional on $T_1 < T_0$, the additional infinitesimal mean
displacement vanishes for $k=0$ (cf.\ \eqref{eq:cor3}) and is strictly
positive and increases with $\alpha$ (which makes the expected hitting
time shorter) for $k \ge 1$.

Conditional under $T_0 < T_1$, and writing $G_\ast$ for the Green
function in this case, we see from formula \eqref{eq:819} in the
Appendix (see also the integrand in \eqref{eq:T4}) that
\begin{align*}
  \frac 1x \frac{\partial}{\partial\alpha} G_\ast(x,y)\Big|_{\alpha = 0} \xrightarrow{x\to 0}
  \frac{4}{\sigma^2(y)}
  \Big(
  (1-y)^2 \int_0^y (1-2z)\psi(z)dz - 2 y(1-y)\int_y^1(1-z)
  \psi(z)dz\Big).
\end{align*}
For $\psi(x) = x^k$, straightforward calculations give that for small
positive $\alpha$ the process $X$ stays in~$dy$ longer than under
neutrality, i.e. for $\alpha=0$, if and only if
$\tfrac 12 ky^k > \sum_{i=0}^{k-1} y^i$. For $k=0,1$, such $y$'s do not
exist. 

\paragraph{Frequency spectrum.}
Changes in the frequency spectrum are often used to infer deviations
from neutral evolution. From Theorem~\ref{T5}, frequencies at $x$ are
higher for small positive~$\alpha$ than for $\alpha=0$ if and only if
$h(x) > h(1)$ for $h(x) := \frac 1x\int_0^x (1-y)\psi(y) dy$. For
$\psi(x) = x^k$, we find that
$h(x) = \frac{1}{k+1}x^k - \frac{1}{k+2}x^{k+1}$. Thus, $h(x) > h(1)$ if and only if
$x^k((k+2)-(k+1)x)>1$. For $k=0$, this is the case for all $x$,
whereas for $k=1$, this is only the case for $x>1/2$. In other words,
high-frequency variants are more abundant under low levels of linear
frequency dependent selection than under neutrality.

\begin{appendix}
\section{Proofs}\label{Sec5}
In what follows, for notational convenience we will suppress the
superscript $\alpha$ and simply write $\mathbf P_x$, $\mathbf E_x$,
$G(x,y)$, \ldots, instead of $\mathbf P^\alpha_x$,
$\mathbf E^\alpha_x$, $G^\alpha(x,y)$,\ldots .

We will express the hitting probability  sxtated in
Theorem~\ref{T1} by the scale function of $X$ (see e.g.\
\cite[p. 192ff]{KarlinTaylor1981}) given by
\begin{equation}\label{eq:scale}
  S(x)  := \int_0^x e^{-2 \alpha \int_0^y \psi(z) dz}dy.
\end{equation}
Then, we have 
\begin{align}
  \label{eq:8312}
  \mathbf P_x(T_1<T_0) 
  & = \frac{S(x)-S(0)}{S(1)-S(0)}.
    \intertext{Moreover, concerning   Theorems~\ref{T2}--\ref{T4},  the expected amount of time which the diffusion
    started in $x$ spends in~$dy$, is $G(x,y)dy$, with the Green function}
    \label{eq:8313}
    G(x,y) & = \begin{cases} \displaystyle 2\mathbf P_x(T_1<T_0) 
      \frac{S(1)-S(y)}{\sigma^2(y)S'(y)}=2\mathbf P_y(T_0<T_1) \frac{S(x)-S(0)}{\sigma^2(y)S'(y)}, & x\leq y,\\[3ex]
      \displaystyle 2\mathbf P_x(T_0<T_1)
      \frac{S(y)-S(0)}{\sigma^2(y)S'(y)} = 2\mathbf P_y(T_1<T_0)
      \frac{S(1)-S(x)}{\sigma^2(y)S'(y)}, &
      x\geq y.\end{cases},
                                            \intertext{When conditioned on $\{T_1<T_0\}$, this changes to}
  \label{eq:8314}
  G^\ast(x,y) 
  & = G(x,y)\frac{\mathbf P_y(T_1<T_0)}{\mathbf P_x(T_1<T_0)},
    \intertext{whereas on $\{T_0<T_1\}$, the Green function is}
    G_\ast(x,y)     \label{eq:8315}
           & = G(x,y)\frac{\mathbf P_y(T_0<T_1)}{\mathbf
             P_x(T_0<T_1)}.
\end{align}
We already gave a proof of Theorem \ref{T1} in Section \ref{Sec4}; here,  we give another short proof using \eqref{eq:scale} and \eqref{eq:8312}.
\begin{proof}[Proof of Theorem \ref{T1}]
  Linearizing \eqref{eq:scale} and using Fubini we obtain
  \begin{equation}
    \label{eq:scaleapprox}
    \begin{aligned}
      S(x) & = x - 2\alpha \int_0^x \int_0^y \psi(z) dz dy + \mathcal
      O(\alpha^2) = x - 2\alpha \int_0^x \int_z^x \psi(z) dy dz +
      \mathcal O(\alpha^2) \\ & = x - 2\alpha \int_0^x (x-y)\psi(y) dy
      + \mathcal O(\alpha^2).
    \end{aligned}
  \end{equation}
  Therefore,
  \begin{equation}
  \label{eq:fixprob}
  \begin{aligned}
    \mathbf P_x(T_1 < T_0) & = \frac{S(x) - S(0)}{S(1) - S(0)} = x -
    2\alpha \int_0^x (x-y) \psi(y) dy + 2\alpha x \int_0^1(1-y)
    \psi(y) dy + \mathcal O(\alpha^2),
  \end{aligned}
    \end{equation}
  and the result follows.
\end{proof}
%
\begin{proof}[Proof of Corollary \ref{cor1}]
  1. Here, $\psi(x) = x^k$ and therefore, the right hand side of
  \eqref{mainres} becomes
  \begin{align*}
    x\int_0^1 (1-y)y^kdy - \int_0^x (x-y)y^k dy = (x-x^{k+2})\Big(\frac{1}{k+1} 
    - \frac{1}{k+2}\Big)  = x(1-x^{k+1}) \frac{1}{(k+1)(k+2)}.
  \end{align*}
  2. Here, we compute for the right hand side of \eqref{mainres}
  \begin{align*}
    x \int_0^1 (1-y)^{k+1} & dy - \int_0^x (1-y - (1-x))(1-y)^k dy 
    \\ & = \frac{x}{k+2} - \frac{1}{k+2} (1 - (1-x)^{k+2}) + \frac{1}{k+1}(1-x)(1-(1-x)^{k+1})
    \\ & = \frac{1}{(k+1)(k+2)}\Big((1-x) - (1-x)^{k+2}\Big).
  \end{align*}
\end{proof}

\begin{proof}[Proof of Theorem \ref{T2}]
  Since $\mathbf E_x[T] = \int_0^1 G(x,y)dy$, we have to approximate
  the Green function, as given in \eqref{eq:8313}. First,
  \begin{equation}
    \label{eq:scaleapprox2}
    \begin{aligned}
      S(x) & = x - 2\alpha \int_0^x (x-y)\psi(y) dy + \mathcal O(\alpha^2),
      \\
      S'(x) & = 1 - 2\alpha \int_0^x \psi(y)dy + \mathcal O(\alpha^2).
    \end{aligned}
  \end{equation}
  Then, for $0 \leq x\leq y \leq 1$ (use \eqref{mainres} for the
  second equality)
  \begin{align}\notag
     G(x,y) & 
              = 2\mathbf P_x(T_1<T_0) \frac{S(1)-S(y)}{\sigma^2(y)S'(y)}
    \\ & \notag= \frac{2}{\sigma^2(y)}
         \Big(x + 2\alpha x\Big(\int_0^1 (1-z)\psi(z) dz - \frac 1x \int_0^x (x-z)\psi(z) dz \Big)\Big)
    \\ & \notag\qquad \cdot
         \Big(1 - y - 2\alpha\Big(\int_0^1(1-z)\psi(z) dz - \int_0^y(y-z)\psi(z)dz
         \Big)\Big(1 + 2\alpha\int_0^y \psi(z)dz\Big) + \mathcal O(\alpha^2)
    \\ & \notag= \frac{2}{\sigma^2(y)} \Big(x(1-y) + 2\alpha \Big(x(1-y) \int_0^y\psi(z) dz 
         - 
         (1-y)\int_0^x (x-z)\psi(z)dz \\ & \notag\qquad \qquad \qquad \qquad \qquad - xy\int_0^1 (1-z)\psi(z) dz 
                                           + x\int_0^y (y-z)\psi(z) dz\Big)\Big)+ \mathcal O(\alpha^2)
    \\ & = \frac{2}{\sigma^2(y)} \Big(x(1-y) + 2\alpha \Big((1-x)(1-y)\int_0^x z \psi(z) dz
         + x(1-y)\int_x^y (1-z)\psi(z) dz \notag
    \\ & \qquad \qquad \qquad \qquad \qquad \qquad \qquad
         -xy \int_y^1 (1-z)\psi(z) dz\Big)\Big)+ \mathcal O(\alpha^2)\label{eq:green}
         \intertext{while for $0 \leq y\leq x \leq 1$}
         G(x,y) & \notag  = 2 \mathbf P_x(T_0<T_1)\frac{S(y)-S(0)}{\sigma^2(y)S'(y)} dy 
    \\ & = \frac{2}{\sigma^2(y)}\notag
         \Big(1-x - 2\alpha x\Big(\int_0^1 (1-z)\psi(z) dz - \frac 1x \int_0^x (x-z)\psi(z) dz \Big)\Big)
    \\ & \notag\qquad \qquad \qquad \cdot
         \Big(y - 2\alpha \int_0^y(y-z)\psi(z)dz \Big)\Big(1 + 2\alpha\int_0^y \psi(z)dz\Big) + \mathcal O(\alpha^2)
    \\ & = \frac{2}{\sigma^2(y)} \Big((1-x)y + 2\alpha \Big((1-x)y\int_0^y \psi(z) dz
         - (1-x) \int_0^y (y-z)\psi(z) dz \notag
    \\ & \qquad \qquad \qquad \qquad \qquad \qquad \qquad -xy
         \int_0^1 (1-z)\psi(z) dz + y \int_0^x (x-z)\psi(z) dz\Big)\Big)+ \mathcal O(\alpha^2) \notag
    \\ & = \frac{2}{\sigma^2(y)} \Big((1-x)y + 2\alpha \Big(
(1-x)(1-y)
         \int_0^y z\psi(z) dz - (1-x)y\int_y^x z\psi(z) dz \notag
    \\ & \qquad \qquad \qquad \qquad \qquad \qquad \qquad - xy\int_x^1 (1-z) \psi(z) dz
         \Big)\Big)+ \mathcal O(\alpha^2). \label{eq:green2}
  \end{align}
  Dividing \eqref{eq:green} and \eqref{eq:green2} by $x$ and letting
  $x\to 0$ gives the result.
\end{proof}

\begin{proof}[Proof of Corollary \ref{cor2}]
  1.  The right hand side of~\eqref{eq:T2} becomes
  \begin{align*}
    & \int_0^1 \int_0^y \frac{1-y}{\sigma^2(y)}(1-z)z^k dz dy - \int_0^1 \int_y^1 \frac{y}{\sigma^2(y)}(1-z)z^k dz dy
    \\ & = \int_0^1 \frac{1-y}{\sigma^2(y)}\Big( \frac{1}{k+1}y^{k+1} - \frac{1}{k+2} y^{k+2}\Big) - \frac{y}{\sigma^2(y)}\Big(\frac{1}{k+1}(1-y^{k+1}) - \frac{1}{k+2}(1-y^{k+2})\Big) dy
    \\ & = \int_0^1 \frac{(1-y)y^{k+1}}{\sigma^2(y)}\Big( \frac{1}{(k+1)(k+2)} + \frac{1}{k+2} (1-y)\Big) 
    \\ & \qquad \qquad \qquad \qquad 
         - \frac{y(1-y^{k+1})}{\sigma^2(y)}\frac{1}{(k+1)(k+2)} + \frac{(1-y)y^{k+2}}{\sigma^2(y)} \frac{1}{k+2}\Big) dy
    \\ & = \int_0^1 \frac{1}{\sigma^2(y)} \Big(\frac{(1-y)y^{k+1}}{k+2} - \frac{y(1-y^k)}{(k+1)(k+2)}\Big) dy.
  \end{align*}
  2.  The right hand side of~\eqref{eq:T2} becomes
  \begin{align*}
    & \int_0^1 \int_0^y \frac{1-y}{\sigma^2(y)}(1-z)^{k+1} dz dy - \int_0^1 \int_y^1 \frac{y}{\sigma^2(y)}(1-z)^{k+1} dz dy
    \\ & = \int_0^1 \frac{1-y}{\sigma^2(y)}\frac{1}{k+2}(1 - (1-y)^{k+2}) 
         - \frac{y}{\sigma^2(y)}\frac{1}{k+2}(1-y)^{k+2}) dy
    \\ & = \frac{1}{k+2} \int_0^1 \frac{1}{\sigma^2(y)}(1-y)(1-(1-y)^{k+1})dy. \quad \Box
  \end{align*}
\end{proof}

\begin{proof}[Proof of Theorem \ref{T3}]
  Recall from \eqref{eq:8314} that for $x\leq y$
  \begin{align*}
    G^\ast(x,y) & = G(x,y) \frac{\mathbf P_y(T_1<T_0)}{\mathbf P_x(T_1<T_0)} = \frac{2}{\sigma^2(y)} \mathbf P_y(T_0 < T_1)\frac{S(y)-S(0)}{S'(y)}.
  \end{align*}
  With $S, S'$ as in \eqref{eq:scaleapprox2}, and with \eqref{eq:fixprob},  we find that
  \begin{equation}
    \label{eq:dsj}
    \begin{aligned}
      \lim_{x\to 0} & \mathbf E^{\alpha\ast}_x[T_1] = \int_0^1
      G^\ast(0,y)dy \\ & = 2 \int_0^1 \frac{1}{\sigma^2(y)}\Big((1 -
      y) + 2\alpha \Big( \int_0^y (y-z)\psi(z)dz - y\int_0^1
      (1-z)\psi(z)dz\Big) \Big) \\ & \qquad \qquad \qquad \qquad \cdot
      \Big(y - 2\alpha \int_0^y (y-z)\psi(z)dz\Big)\cdot \Big(1 +
      2\alpha \int_0^y \psi(z)dz\Big) \frac{1}{\sigma^2(y)}dy +
      \mathcal O(\alpha^2) \\ & = 2\int_0^1 \frac{(1-y)y}{\sigma^2(y)}
      dy + 4\alpha \int_0^1 \Big(\frac{y(1-y)}{\sigma^2(y)} \int_0^y
      \psi(z)dz + \frac{y}{\sigma^2(y)} \int_0^y (y-z) \psi(z)dz \\ &
      \qquad \qquad \qquad \qquad -\frac{y^2}{\sigma^2(y)}
      \int_0^1(1-y+y-z) \psi(z)dz - \frac{1-y}{\sigma^2(y)}\int_0^y
      (y-z) \psi(z)dz \Big)dy + \mathcal O(\alpha^2) \\ & = 2\int_0^1
      \frac{(1-y)y}{\sigma^2(y)} dy + 4\alpha \int_0^1
      \Big(\frac{y(1-y)}{\sigma^2(y)} \int_0^y \psi(z)dz -
      \frac{(1-y)^2}{\sigma^2(y)} \int_0^y (y-z) \psi(z)dz \\ & \qquad
      \qquad \qquad \qquad - \frac{y^2}{\sigma^2(y)} \int_y^1(1-z)
      \psi(z)dz - \frac{y^2(1-y)}{\sigma^2(y)}\int_0^y \psi(z)dz\Big)
      dy+ \mathcal O(\alpha^2) \\ & = 2\int_0^1
      \frac{(1-y)y}{\sigma^2(y)} dy + 4\alpha \int_0^1
      \Big(\frac{(1-y)^2}{\sigma^2(y)} \int_0^y z\psi(z)dz -
      \frac{y^2}{\sigma^2(y)} \int_y^1(1-z) \psi(z)dz \Big)dy +
      \mathcal O(\alpha^2)
    \end{aligned}
  \end{equation}
  and we are done.
\end{proof}

\begin{proof}[Proof of Corollary \ref{cor3}]
  1. For $\psi(x) = x^k$, the right hand side of \eqref{eq:T3} becomes
  \begin{align*}
    & \int_0^1 \frac 1{\sigma^2(y)} \left( (1-y)^2 \int_0^y
      z^{k+1} dz -y^2 \int_y^1(1-z)z^kdz \right)dy
    \\ & = \int_0^1 \frac 1{\sigma^2(y)} \left( (1-y)^2 \frac{1}{k+2} y^{k+2} -y^2 \Big( \frac{1}{k+1}(1 - y^{k+1}) - \frac{1}{k+2}(1-y^{k+2})\Big)\right)dy
    \\ & = \int_0^1 \frac 1{\sigma^2(y)} \left( (1-y)^2 y^{k+2} \frac{1}{k+2} -y^2(1-y^{k+1})\frac{1}{(k+1)(k+2)} + y^{k+3}(1-y) \frac{1}{k+2}\right)dy
    \\ & = \int_0^1 \frac{1}{\sigma^2(y)} \Big(\frac{(1-y)y^{k+2}}{k+2} - \frac{y^2(1-y^{k+1})}{(k+1)(k+2)}\Big)dy.
  \end{align*}
  2. For $\psi(x) = (1-x)^k$, the right hand side of \eqref{eq:T3}
  becomes
  \begin{align*}
    & \int_0^1 \frac 1{\sigma^2(y)} \left( (1-y)^2 \int_0^y
      z(1-z)^k dz -y^2 \int_y^1(1-z)^{k+1}dz \right)dy
    \\ & \stackrel{y\to 1-y, z\to 1-z} = \int_0^1 \frac 1{\sigma^2(1-y)} \left( y^2 \int_y^1
         (1-z)z^k dz -(1-y)^2 \int_0^y z^{k+1}dz \right)dy
    \\ & = - \int_0^1 \frac{1}{\sigma^2(1-y)} \Big(\frac{(1-y)y^{k+2}}{k+2} - \frac{y^2(1-y^{k+1})}{(k+1)(k+2)}\Big)dy.
    \\ & = - \int_0^1 \frac{1}{\sigma^2(y)} \Big(\frac{y(1-y)^{k+2}}{k+2} - \frac{(1-y)^2(1-(1-y)^{k+1})}{(k+1)(k+2)}\Big)dy,
  \end{align*}
  where in the second equality we have used the display from part 1 of the proof.
\end{proof}
\begin{remark}\label{mustar}
  We append here the calculation (for small $\alpha$) of the
  infinitesimal mean displacement of the diffusion process $X$
  conditioned hit 1; this was announced in Section~\ref{Sec4} as an
  additional explanation of the monotone increase of
  $\mathbf E_{0+}^{\alpha \ast}[T_1]$ for small $\alpha$. Recall that
  for the solution of \eqref{eq:SDE}, conditioned to hit~1, $\mu$
  becomes (see e.g.\ \citealp[p.\ 263]{KarlinTaylor1981})
  \begin{align*}
    \mu^\ast(x) & = \alpha \mu(x) + \frac{S'(x)}{S(x) -S(0)} \sigma^2(x) 
                  = \alpha \mu(x) + \Big(\frac{1}{x}\frac{1-2\alpha \int_0^x \psi(y) dy}{1-\frac{2\alpha}{x} \int_0^x(x-y) \psi(y) dy}\Big) \sigma^2(x) + \mathcal O(\alpha^2)
    \\ & = \alpha \mu(x) + \frac{1}{x}\Big(1 - 2\alpha\Big( \int_0^x \psi(y) dy - \int_0^x\frac{x-y}{x}\psi(y) dy\Big)\Big) \sigma^2(x) + \mathcal O(\alpha^2)                  
    \\ & = \frac 1x \sigma^2(x) + \alpha \Big( \psi(x) - \frac{2}{x^2} \int_0^x y\psi(y) dy\Big) \sigma^2(x) + \mathcal O(\alpha^2).
  \end{align*}
  In particular, if $\psi(x) = x^k$, 
  \begin{align*}
    \mu^\ast(x) & = \frac 1x \sigma^2(x) + \alpha \Big( x^k - \frac{2}{k+2}x^k\Big)\sigma^2(x) + \mathcal O(\alpha^2)
                  = \frac 1x \sigma^2(x) + \alpha \frac{k}{k+2} x^k \sigma^2(x) + \mathcal O(\alpha^2).
  \end{align*}
  This shows that the additional infinitesimal mean displacement increases with $\alpha$.
  Moreover, the additional infinitesimal mean displacement vanishes for $k=0$,
  which explains \eqref{eq:cor3}. 
\end{remark}
\begin{proof}[Proof of Theorem~\ref{T4}]
  We start by writing the Green function for the diffusion $X$,
  conditional to hit~0, which is from \eqref{eq:8315}
  \begin{align}\label{eq:817}
    G_\ast(x,y) = G(x,y) \frac{\mathbf P_y(T_0<T_1)}{\mathbf P_x(T_0<T_1)}.
  \end{align}
  Note that for $x\leq y$
  \begin{equation}
    \label{eq:818}
    \begin{aligned}
      \frac{\mathbf P_y(T_0<T_1)}{\mathbf P_x(T_0<T_1)} & = \frac{1-y
        - 2\alpha \Big(y \int_0^1 (1-z)\psi(z)dz - \int_0^y
        (y-z)\psi(z) dz\Big)}{1-x - 2\alpha
        \Big(x\int_0^1(1-z)\psi(z)dz - \int_0^x(x-z)\psi(z) dz\Big)} +
      \mathcal O(\alpha^2) \\ & = \frac{1-y}{1-x} + 2\alpha \Big(
      \frac{(1-y)x}{(1-x)^2}\int_0^1(1-z)\psi(z)dz -
      \frac{(1-y)}{(1-x)^2} \int_0^x(x-z)\psi(z) dz \\ & \qquad \qquad
      \qquad \qquad - \frac{y}{1-x} \int_0^1 (1-z)\psi(z)dz +
      \frac{1}{1-x}\int_0^y (y-z)\psi(z) dz\Big) + \mathcal
      O(\alpha^2) \\ & = \frac{1-y}{1-x} + 2\alpha \Big(
      \frac{x-y}{(1-x)^2}\int_0^1(1-z)\psi(z)dz -
      \frac{(1-y)}{(1-x)^2} \int_0^x(x-z)\psi(z) dz \\ & \qquad \qquad
      \qquad \qquad \qquad \qquad \qquad \qquad \qquad \qquad +
      \frac{1}{1-x}\int_0^y (y-z)\psi(z) dz\Big) + \mathcal
      O(\alpha^2) \\ & =\frac{1-y}{1-x} - 2\alpha \Big(
      \frac{1-y}{(1-x)^2}\int_x^y(z-x)\psi(z)dz +
      \frac{y-x}{(1-x)^2}\int_y^1(1-z)\psi(z)dz\Big) + \mathcal
      O(\alpha^2).
    \end{aligned}
  \end{equation}
  Combining the last equality with \eqref{eq:817} and \eqref{eq:green},
  \begin{equation}
    \label{eq:819}
    \begin{aligned}
      \frac{1}{x} \mathbf E^\alpha_{x\ast}[T_0]
      & 
      = \frac{1}{x} \int_0^x \frac{2}{\sigma^2(y)} \mathbf
      P_y(T_0<T_1) \frac{S(y)-S(0)}{S'(y)}dy + \frac{1}{x} \int_x^1
      G(x,y) \frac{\mathbf P_y(T_0<T_1)}{\mathbf P_x(T_0<T_1)}dy \\ &
      = \frac{2}{x}\int_0^x \frac{y}{\sigma^2(y)} dy \\ & \qquad
      \qquad + \int_x^1 \frac{2}{\sigma^2(y)}\Big(1-y + 2\alpha
      \Big((1-y)\int_0^y (1-z)\psi(z)dz -
      y\int_y^1(1-z)\psi(z)dz\Big)\Big) \\ & \qquad \qquad \qquad
      \cdot \Big(1-y - 2\alpha\Big((1-y)\int_0^yz\psi(z)dz + y
      \int_y^1 (1-z)\psi(z) dz\Big)\Big)dy + \mathcal O(\alpha^2, x)
      \\ & = \frac 2x \int_0^x \frac{y}{\sigma^2(y)} dy + 2\int_x^1
      \frac{(1-y)^2}{\sigma^2(y)}dy \\ & \qquad + 4\alpha \int_x^1
      \Big( \frac{(1-y)^2}{\sigma^2(y)} \int_0^y \big((1-z)\psi(z) -
      z\psi(z)\big) dz \\ & \qquad \qquad \qquad \qquad - 
      \frac{y(1-y)}{\sigma^2(y)} \int_y^1 \big((1-z)\psi(z) +
      (1-z)\psi(z)\big)dz\Big) dy + \mathcal O(\alpha^2, x)
    \end{aligned}
  \end{equation}
  and the result follows.
\end{proof}

\begin{proof}[Proof of Corollary \ref{cor4}]
  1. For $\psi(x) = x^k$, the right hand side of \eqref{eq:T4} gives
  \begin{align*}
    & \int_0^1 \frac 1{\sigma^2(y)} \left( (1-y)^2 \int_0^y
      (1-2z)z^{k} dz - 2y(1-y) \int_y^1(1-z)z^kdz \right)dy
    \\ & = \int_0^1 \frac{1-y}{\sigma^2(y)} \left( (1-y) \Big(\frac{1}{k+1} y^{k+1} - \frac{2}{k+2} y^{k+2}\Big) - 2y \Big( \frac{1}{k+1}(1 - y^{k+1}) - \frac{1}{k+2}(1-y^{k+2})\Big)\right)dy
    \\ & = \int_0^1 \frac{1-y}{\sigma^2(y)} \left( \big(y(1-y^{k+2}) - (1-y)y^{k+2}\big)\frac{2}{k+2} + \big( (1-y)y^{k+1} -2y(1-y^{k+1})\big) \frac{1}{k+1}\right)dy
    \\ & = \int_0^1 \frac{y(1-y)}{\sigma^2(y)} \Big(\frac{2(1-y^{k+1})}{k+2} - \frac{2 - y^k - y^{k+1}}{k+1}\Big)dy = \int_0^1 \frac{y(1-y)}{\sigma^2(y)} \Big(\frac{k(1-y^{k+1})}{(k+1)(k+2)} - \frac{1 - y^k}{k+1}\Big)dy.
  \end{align*}
  2. For $\psi(x) = (1-x)^k$, the right hand side of \eqref{eq:T4}
  becomes
  \begin{align*}
    & \int_0^1 \frac 1{\sigma^2(y)} \left( (1-y)^2 \int_0^y
      (1-2z)(1-z)^k dz - 2y(1-y) \int_y^1(1-z)^{k+1}dz \right)dy
    \\ & \stackrel{y\to 1-y, z\to 1-z} = \int_0^1 \frac y{\sigma^2(1-y)} \left( - y \int_y^1
         (1-2z)z^k dz -2(1-y) \int_0^y z^{k+1}dz \right)dy
    \\ & = - \int_0^1 \frac{y}{\sigma^2(1-y)} \Big( \frac{y(1-y^{k+1})}{k+1} + \frac{2(1-y)y^{k+2} - 2y(1-y^{k+2})}{k+2}\Big)dy.
    \\ & = - \int_0^1 \frac{y^2}{\sigma^2(1-y)} \Big(\frac{1-y^{k+1}}{k+1} - \frac{2(1-y^{k+1})}{k+2}\Big)dy = 
         \int_0^1 \frac{y^2(1-y^{k+1})}{\sigma^2(1-y)}\frac{k}{(k+1)(k+2)}dy.
    \\ & = 
         \int_0^1 \frac{(1-y)^2(1-(1-y)^{k+1})}{\sigma^2(y)}\frac{k}{(k+1)(k+2)}dy.
  \end{align*}
\end{proof}

\begin{proof}[Proof of Theorem~\ref{T5}]
 Because of  \eqref{mainres}, $f^\alpha$
  satisfies
  \begin{align*}
    f^\alpha(x) & = \frac{1}{\sigma^2(x)} \Big(1 + 2\alpha \int_0^x \psi(y)dy\Big)\Big(1 - x - 2\alpha \Big(x \int_0^1(1-y)\psi(y) dy - \int_0^x(x-y)\psi(y) dy\Big)\Big) + \mathcal O(\alpha^2)
    \\ & = \frac{1-x}{\sigma^2(x)} + \frac{2\alpha}{\sigma^2(x)}\Big((1-x)\int_0^x \psi(y)dy - x \int_0^1(1-y)\psi(y) dy + \int_0^x(x-y)\psi(y) dy\Big)+ \mathcal O(\alpha^2)
    \\ & = \frac{1-x}{\sigma^2(x)} + \frac{2\alpha x}{\sigma^2(x)}\Big(\frac 1x \int_0^x (1-y)\psi(y)dy - \int_0^1(1-y)\psi(y) dy\Big)+ \mathcal O(\alpha^2).
  \end{align*}
\end{proof}

\begin{proof}[Proof of Corollary~\ref{cor5}]
  1. We compute
  \begin{align*}
    \frac 1x \int_0^x (1-y)y^k dy - \int_0^1(1-y)y^k dy & = \frac{1}{k+1}x^k - \frac{1}{k+2}x^{k+1} - \frac{1}{(k+1)(k+2)}
    \\ & = - \frac{1 - x^k(1 + (1-x)(k+1))}{(k+1)(k+2)}.
  \end{align*}
  2. We compute
  \begin{align*}
    \frac 1x \int_0^x (1-y)^{k+1} dy - \int_0^1(1-y)^{k+1} dy & = \frac{1}{k+2}\Big(\frac 1x (1-(1-x)^{k+2}) - 1 \Big).
  \end{align*}
\end{proof}
\end{appendix}

\subsubsection*{Acknowledgments}
This research was supported by the DFG priority program SPP 1590, and in
particular through grant Pf-672/6-1 to PP, and Wa-967/4-1 to AW. We
thank Arne Traulsen for pointing out the relevant work of
\cite{LachmannMafessoni2015}. Also, we are grateful to the Editor and two referees for suggestions that helped to improve the presentation.

\end{document}